\documentclass[a4paper,11pt]{amsart}
\usepackage[T1]{fontenc}
\usepackage[utf8]{inputenc}
\usepackage{xifthen, xfrac, mathrsfs, url, amssymb} 
\usepackage[colorinlistoftodos]{todonotes}
\newtheorem{thm}{Theorem}
\newtheorem{alg}[thm]{Algorithm}

\newtheorem{lem}[thm]{Lemma}
\newtheorem{prop}[thm]{Proposition}
\newtheorem{cor}[thm]{Corollary}
\theoremstyle{remark}
\newtheorem*{rem}{Remark}
\newenvironment{poc}[1][]{\begin{proof}[\ifthenelse{\equal{#1}{}}{Proof of correctness}{Proof of correctness of #1}]}{\end{proof}}
\newcommand{\gb}{\mathfrak{b}}
\newcommand{\go}{\mathfrak{o}}
\newcommand{\gp}{\mathfrak{p}}
\newcommand{\gq}{\mathfrak{q}}
\newcommand{\gr}{\mathfrak{r}}
\newcommand{\gs}{\mathfrak{s}}
\newcommand{\gt}{\mathfrak{t}}
\newcommand{\Fq}{\FF_q}

\newcommand{\FF}{\mathbb{F}}
\newcommand{\ZZ}{\mathbb{Z}}

\newcommand{\GB}{\mathfrak{B}}
\newcommand{\SA}{\mathscr{A}}
\newcommand{\SB}{\mathscr{B}}
\newcommand{\SD}{\mathcal{D}}
\newcommand{\SL}{\mathscr{L}}

\newcommand{\K}{K}
\newcommand{\setS}{\mathfrak{S}}
\newcommand{\setT}{\mathfrak{T}}
\newcommand{\Sprime}{\setS'}
\newcommand{\SmS}{\setS\setminus\Sprime}
\newcommand{\Ell}{\mathcal{L}}
\newcommand{\card}[1]{|#1|}

\newcommand{\units}[1]{#1^\times}
\newcommand{\squares}[1]{#1^{\times2}}
\newcommand{\sing}[1]{E_{#1}}
\newcommand{\singXS}{\sing{\setS}}
\newcommand{\Sing}[1]{\mathbb{E}_{#1}}
\newcommand{\SingX}{\Sing{}}
\newcommand{\SingXS}{\Sing{\setS}}

\newcommand{\SingXSprime}{\Sing{\Sprime}}
\newcommand{\st}{\mathrel{\mid}}
\DeclareMathOperator{\Div}{Div}
\newcommand{\DivX}{\Div X}
\DeclareMathOperator{\Pic}{Pic}
\newcommand{\PicX}{\Pic X}
\newcommand{\PicZeroX}{\Pic^0 X}

\newcommand{\PPic}[1]{\sfrac{\Pic #1}{2\Pic #1}}
\newcommand{\PPicX}{\PPic{X}}
\DeclareMathOperator{\rank}{rk}

\newcommand{\rk}[1][]{\ifthenelse{\equal{#1}{}}{\rank_{2}}{\rank_{#1}}}
\newcommand{\class}[2][]{\ifthenelse{\equal{#1}{}}{[#2]}{[#2]_#1}}   
\DeclareMathOperator{\lin}{span}
\DeclareMathOperator{\dv}{div}
\newcommand{\divX}{\dv_X}
\DeclareMathOperator{\ord}{ord}
\newcommand{\un}[1][]{\ifthenelse{\equal{#1}{}}{\units{\K}}{\units{#1}}}
\newcommand{\sqgd}[1][]{\ifthenelse{\equal{#1}{}}{\sfrac{\un}{\squares{\K}}}{\sfrac{\units{#1}}{\squares{#1}}}}
\newcommand{\term}[1]{\emph{#1}}
\newcommand{\la}{\lambda}
\newcommand{\lap}{\la_\gp}
\newcommand{\even}{\equiv 0\pmod{2}}
\newcommand{\odd}{\equiv 1\pmod{2}}

\newcommand{\subgrp}[1]{\langle#1\rangle}

\newcommand{\Legendre}[2]{\genfrac{(}{)}{}{}{#1}{#2}}

\def\clap#1{\hbox to 0pt{\hss#1\hss}}

\newwrite\refs
\openout\refs=\jobname.refs
\makeatletter
\renewcommand\@setref[3]{%
        \ifx#1\relax
                \write\refs{'#3' \thepage\space undefined}%
                \protect \G@refundefinedtrue
                \nfss@text{\reset@font\bfseries ??}%
                \@latex@warning{Reference `#3' on page \thepage\space undefined}%
        \else
                \write\refs{'#3' \thepage\space \expandafter\@secondoftwo#1}%
                \expandafter#2#1\null
        \fi
}
\makeatother
\title{Computing singular elements modulo squares}
\author{Przemys{\l}aw Koprowski}
\address{Institute of Mathematics, University of Silesia, Bankowa 14, 40-007 Katowice, Poland } 
\email{przemyslaw.koprowski@us.edu.pl}
\begin{document}
\begin{abstract}
The group of singular elements was first introduced by Helmut Hasse and later it has been studied by numerous authors including such well known mathematicians as: Cassels, Furtw\"{a}ngler, Hecke, Knebusch, Takagi and of course Hasse himself; to name just a few. The aim of the present paper is to present algorithms that explicitly construct groups of singular and $S$-singular elements (modulo squares) in a global function field.
\end{abstract}
\maketitle

\section{Introduction}
Let $\K$ be a global function field of odd characteristic and $\Fq$ be its full field of constants, where $q$ is a power of an odd prime. The set of classes of discrete valuations on~$\K$ can be viewed as a smooth complete curve~$X$ over~$\Fq$. Assume that $\setS$ is a finite (possibly empty) subset of~$X$. An element $\la\in \un$ is called \term{$\setS$-singular} if it has even valuations everywhere outside~$\setS$. We denote
\[
\singXS = \{\la\in \un\st \la\text{ is $\setS$-singular}\}
\]
the group of $\setS$-singular elements of~$\K$. This group is a union of cosets of~$\squares{\K}$, hence we take a quotient group and denote it
\[
\SingXS := \sfrac{\singXS}{\squares{\K}} = \{\la\in \sqgd\st \ord_\gp\la\even\text{ for }\gp\in X\setminus\setS\}.
\]
If $\setS$ is empty we just write~$\SingX$, rather than a bit awkward $\Sing{\emptyset}$. In particular the group $\singXS$ can be expressed as the extension of~$\SingXS$ by~$\squares{\K}$.  The group~$\SingX$ is sometimes called \term{$2$-Selmer group} (see \cite{Cohen2000, DV2018, Lemmermeyer2006}). The quotient group~$\SingXS$ is a finite elementary $2$-group, hence a finitely dimensional vector space over~$\FF_2$. As such it can be explicitly represented by its (finite) basis.

The aim of this paper is to present algorithms for constructing such bases. Of course, this problem is not completely new. For example, H.~Cohen in \cite{Cohen2000} describes an algorithm for computing the group~$\SingX$ in case when~$\K$ is a \emph{number} field. The main contribution of the present paper is twofold. First, we propose a use of auxiliary sets of places ``compatible'' (the term is explained later) with the constructed bases, that substantially simplify computation of coordinates of any given singular (respectively $\setS$-singular) element of~$\un$ (see Propositions~\ref{prop:coords} and~\ref{prop:S-coords}). Secondly, since computation of a Picard group may be a time consuming process, in Section~\ref{sec:random} we introduce a randomized method for constructing singular elements without knowing the Picard group of~$X$.

Throughout this paper we use the following notation: for a place~$\gp$ of~$X$ by $\ord_\gp$ we denote the associated valuation. Observe that the parity of valuation is fixed throughout any given square class of~$K$, hence we tend to treat $\ord_\gp$ as a function from $\sqgd$ to~$\ZZ_2$. In addition, for $\la\in \sqgd$ by $\Legendre{\la}{\gp}$ we denote the natural generalization of the Legendre symbol to the whole square class group of~$K$, that is
\[
\Legendre{\la}{\gp} =
\begin{cases}
1 & \text{if $\la$ is a square in $\K_\gp$ (in particular $\ord_\gp\la$ is even),}\\
0 & \text{if $\ord_\gp\la$ is odd,}\\
-1 & \text{if $\ord_\gp\la$ is even but $\la$ is not a square in~$K_\gp$.}
\end{cases}
\]
Next, for a divisor $\SD\in \DivX$, by $\class{\SD}$ we denote its class in the Picard group $\PicX$ and by $\dim\SD$ the dimension of the Riemann--Roch space $\Ell(\SD)$ of~$\SD$, where
\[
\Ell(\SD) := \bigl\{ \lambda\in \K\st \divX\lambda\geq -\SD\bigr\}\cup \{0\}.
\]
We will frequently use the fact that if $G$ is an abelian group, then $\sfrac{G}{2G}$ has a natural structure of an $\FF_2$-linear space. The dimension of this space is called the $2$-rank of~$G$ and denoted $\rk G$. Furthermore, when dealing with a finite set of places $\gp_1, \dotsc, \gp_n\in X$, we will consider the cosets modulo $2\PicX$ of the classes $\class{\gp_1}, \dotsc, \class{\gp_n}\in \PicX$. In order to simplify wording, we shall write ``$\gp_1, \dots, \gp_n$ are linearly independent (or form a basis) in $\PPicX$'', meaning actually that these are the cosets modulo $2\PicX$ of their classes that are linearly independent (respectively form a basis) in $\PPicX$. We hope that the reader will excuse us this abuse of terminology, for sake of clarity of exposition. 

The algorithms presented in this paper rely on some known procedures, like computation of the Riemann--Roch space of a divisor (see e.g. \cite{Hess2002}) or construction of the Picard group of~$X$ (see e.g. \cite{Hess1999, Hess2007}). Moreover, we assume availability of standard linear algebra routines to deal with vector spaces over fields and lattices over~$\ZZ$.

\section{Group of singular elements}
Let $\GB = \{\gb_1, \dotsc, \gb_n\}\subset X$ be a set of places and $\SB := \{\beta_1, \dotsc, \beta_n\}\subset \SingX$ be a set of square classes of singular elements. The two sets are said to be \term{compatible} if
\[
\Legendre{\beta_i}{\gb_i} = -1
\qquad\text{and}\qquad
\Legendre{\beta_i}{\gb_j} = 1
\]
for all pairs of distinct indices $i,j\leq n$. In other words, $\beta_i$ is a local square at all $\gp_j$ but~$\gp_i$, where it is a non-square.

\begin{prop}\label{prop:compatible_bases}
Let $\GB\subset X$ and $\SB\subset \SingX$ be two compatible sets. The following conditions are equivalent:
\begin{itemize}
\item $\SB$ is linearly independent in~$\SingX$;
\item $\GB$ is linearly independent in $\PPicX$.
\end{itemize}
In particular, $\SB$ is a basis of~$\SingX$ if and only if $\GB$ is a basis of~$\PPicX$.
\end{prop}

For the proof of the proposition see \cite[Section~3]{CK2020}. Using a pair of compatible bases has an evident advantage over working with a single basis. It is easy to get coordinates with respect to one basis from Legendre symbols computed against the other one. The following proposition is a special case of \cite[Proposition~5]{CK2020}. 

\begin{prop}\label{prop:coords}
Let $\SB = \{\beta_1, \dotsc, \beta_n\}$ be a basis of~$\SingX$ and $\GB = \{\gb_1, \dotsc, \gb_n\}$ be a compatible set of places.
\begin{enumerate}
\item If $\la\in \un$ is a singular element, then its coordinates $\varepsilon_1, \dotsc, \varepsilon_n\in \{0,1\}$ with respect to~$\SB$ are given by a formula
\[
(-1)^{\varepsilon_i} = \Legendre{\la}{\gb_i},\qquad\text{for }i\in \{1,\dotsc, n\}.
\]
\item If $\gp\in X$ is a place, then the coordinates $e_1, \dotsc, e_n\in \{0,1\}$ of~$\class{\gp} + 2\PicX$ in $\PPicX$ with respect to the basis $\GB$ satisfy the conditions
\[
(-1)^{e_j} = \Legendre{\beta_j}{\gp},\qquad\text{for }j\in \{1,\dotsc, n\}.
\]
\end{enumerate}
\end{prop}

We may now present our first algorithm, that constructs a basis of the group of singular elements (modulo squares) from a compatible basis of $\PPicX$. Recall that a finitely generated abelian group~$G$ is presented in \term{Smith Normal Form} (SNF) if it is given as a direct sum of cyclic groups
\[
G = C_1\oplus \dotsb\oplus C_k,
\]
where the order of each~$C_i$ divides the order of its successors. 

\begin{alg}\label{alg:singular_basis}
Given a finite set $\GB = \{\gb_1, \dotsc, \gb_n\}\subset X$ of places that forms a basis of $\PPicX$, this algorithm constructs a compatible basis~$\SB$ of~$\SingX$.
\begin{enumerate}
\item Initialize $\SA := \{\zeta\}$, where $\zeta\in \units{\Fq}$ is a non-square constant.
\item\label{st:Pic0_SNF} Let $\subgrp{\class{\SD_1}}\oplus\dotsb \oplus \subgrp{\class{\SD_m}}$ be SNF presentation of $\PicZeroX$, for some divisors $\SD_1,\dotsc, \SD_m$.
\item Find the minimal index $k\leq m$ such that $\class{\SD_k}$ has an even order or set $k := m + 1$ if all the orders are odd.
\item\label{st:EE_X_loop} For every index $i\in  \{k, \dotsc, m\}$ proceed as follows:
	\begin{enumerate}
	\item Let $d$ be be the order of $\class{\SD_i}$ in $\PicZeroX$.
	\item Find a nonzero element~$\alpha$ in the Riemann--Roch space $\Ell(d\cdot \SD_i)$.
	\item Append $\alpha$ to~$\SA$.
	\end{enumerate}
\item\label{st:singular_coords} Using Proposition~\ref{prop:coords} for every element $\alpha_i\in \SA$, $i\leq n$ find its coordinates $\varepsilon_{i,1},\dotsc, \varepsilon_{i,n}\in \FF_2$ with respect to the sought basis~$\SB$ compatible with~$\GB$.
\item Build a change of basis matrix $( e_{i,j} ) := ( \varepsilon_{i,j} )^{-1}$.
\item Construct the basis $\SB := \{ \beta_1, \dotsc, \beta_n \}$ setting
\[
\beta_i := \alpha_1^{e_{i,1}}\dotsm \alpha_n^{e_{i,n}},\qquad\text{for }i\in \{1,\dotsc, n\}.
\]
\item\label{st:output_SB} Output $\SB$
\end{enumerate}
\end{alg}

\begin{poc}
The first part of the algorithm (steps \ref{st:Pic0_SNF}--\ref{st:EE_X_loop}) is nothing else but an adaptation of \cite[Definition~5.2.7]{Cohen2000} to the case of function fields. On the other hand, steps \ref{st:singular_coords}--\ref{st:output_SB} are just a standard change of basis. Thus, the correctness of the algorithm follows immediately from Propositions~\ref{prop:compatible_bases} and~\ref{prop:coords}.
\end{poc}

\begin{rem}
Observe that the $2$-rank of~$\SingX$ equals the number of generators of $\PicZeroX$ in SNF that have even orders. Therefore, when $\PicZeroX$ has odd order, then $k$ is set to $m+1$ and so the loop in step~\eqref{st:EE_X_loop} is empty.
\end{rem}

A general idea how to construct a basis of~$\SingX$ is now clear. First we find a set~$\GB$ of places whose classes form a basis of $\PPicX$, then we build a compatible basis~$\SB$. For sake of completeness let us write it down explicitly. In the next algorithm we assume that we have a method for constructing a set of divisors whose classes generate~$\PicX$. Such an algorithm is described e.g. in \cite{Hess2007}.

\begin{alg}\label{alg:singular_group}
Given a global function field~$\K$, this algorithm constructs a basis \textup(over $\FF_2$\textup) of the group~$\SingX$ of singular elements modulo squares.
\begin{enumerate}
\item Find divisors $\SD_1, \dotsc, \SD_n$ whose classes generate~$\PicX$.
\item Find a set of indices~$J\subset \{1,\dotsc, n\}$ such that the set $\{[\SD_j]+2\PicX\st j\in J\}$ generate the quotient group $\PPicX$.
\item Let $\{ \gp_{j,1}, \dotsc, \gp_{j,k_j}\}$ be the support of $\SD_j$ for every $j\in J$.
\item Using linear algebra find a maximal linearly independent \textup(in $\PPicX$\textup) subset~$\GB$ of $\{ \gp_{j,i}\st j\in J, i\leq k_j\}$.
\item Execute Algorithm~\ref{alg:singular_basis} to construct a basis~$\SB$ compatible with~$\GB$.
\item Output~$\SB$.
\end{enumerate}
\end{alg}

Correctness of the above algorithm follows from the preceding discussion. 

\section{Group of $\setS$-singular elements}
Now we turn our attention to a group of $\setS$-singular elements for some finite (generally nonempty) subset $\setS\subset X$. This task is substantially harder. We begin with a proposition describing the structure of the group~$\SingXS$.

\begin{prop}\label{prop:S-singular_basis}
Let $\SB = \{\beta_1, \dotsc, \beta_n\}$ be a basis of~$\SingX$ and $\GB = \{\gb_1, \dotsc, \gb_n\}$ a compatible set of places. Further let $\setS\subset X$ be a finite set disjoint with~$\GB$ and $\Sprime\subseteq \setS$ be a maximal subset of~$\setS$ linearly independent in $\PPicX$. Then:
\begin{enumerate}
\item For every place $\gp\in\SmS$, there is a $\bigl(\Sprime\cup\{\gp\}\bigr)$-singular element $\lap\in\un$ such that 
\[
\ord_\gp\lap\odd
\qquad\text{and}\qquad
\Legendre{\lap}{\gb} = 1\quad\text{for every }\gb\in \GB.
\]
\item The set $\SB\cup\{\lap\st\gp\in \SmS\}$ is a basis of~$\SingXS$.
\end{enumerate}
\end{prop}

\begin{proof}
Without loss of generality we may assume that $\setS$ is not empty. If the classes of places in~$\setS$ are linearly independent modulo~$2\PicX$, then $\setS = \Sprime$, hence the first assertion holds vacuously and the second one is an immediate consequence of \cite[Lemma~2.3]{CKR2020}, which says that in such case $\SingXS = \SingX$.

Assume now that $\Sprime = \{\gs_1,\dotsc, \gs_m\}\subsetneq \setS$ and fix a place $\gp\in \SmS$. Then the coset of $\class{\gp}$ modulo $2\PicX$ is a linear combination (over~$\FF_2$) of cosets of $\gs_1,\dotsc, \gs_m$. Hence there are: $\varepsilon_1, \dotsc, \varepsilon_m\in \{0,1\}$, a divisor~$\SD$ and an element $\la\in \un$ such that
\[
\divX\la = \gp + \sum_{i\leq m}\varepsilon_i\gs_i + 2\SD.
\]
It is then clear that $\ord_\gp\la\odd$ and $\la$ is $\bigl(\Sprime\cup\{\gp\}\bigr)$-singular. Set $\lap := \la\cdot \beta_1^{e_1}\dotsm \beta_n^{e_n}$, where $\Legendre{\la}{\gb_i} = (-1)^{e_i}$. Then $\lap$ remains $\bigl(\Sprime\cup\{\gp\}\bigr)$-singular since $\beta_1, \dotsc, \beta_n\in \SingX$. Moreover, $\lap$ is a local square at each $\gb_i\in \GB$. This proves the first assertion.

In order to prove the second assertion let $\SmS = \{\gp_1, \dotsc, \gp_n\}$. Observe that the square classes of $\la_{\gp_1}, \dotsc, \la_{\gp_n}$ are linearly independent in $\sqgd$, since otherwise we would have $1 = \la_{\gp_1}^{e_1}\dotsm \la_{\gp_n}^{e_n}$ for some $e_1, \dotsc, e_n\in \{0,1\}$, not all equal zero. But then $1$ would have a nonzero valuation at some~$\gp_i$, which is impossible. 

On the other hand, any nontrivial element of the subgroup of~$\sqgd$ generated by $\la_{\gp_1}, \dotsc, \la_{\gp_n}$ has an odd valuation at some~$\gp_i$. Therefore this group intersects $\SingX$ only at~$\{1\}$. Furthermore by \cite[Lemma~2.3]{CKR2020} we have $\rk\SingX = \rk \SingXSprime$. Thus \cite[Lemma~2.5]{CKR2018} together with \cite[Proposition~2.3]{CKR2018} yield
\begin{align*}
\rk\SingXS 
&= \rk\Pic(X\setminus\Sprime)  + \card{\setS}\\
&\hspace{2em}- \rk\lin_{\FF_2}\bigl\{ \class{\gp}+2\Pic(X\setminus\Sprime)\st \gp\in \SmS \bigr\}\\
&= (\rk\SingXSprime - \card{\Sprime}) + \card{\setS} - 0\\
&= \rk\SingX + \card{\setS\setminus\Sprime}\\
&= \rk\left( \SingX\oplus \lin_{\FF_2}\{\la_{\gp_1}, \dotsc, \la_{\gp_n}\} \right).
\end{align*}
This proves that $\SB\cup\{\lap\st\gp\in \SmS\}$ is indeed a basis of~$\SingXS$.
\end{proof}

\begin{prop}\label{prop:S-coords}
Keep the assumptions of the previous proposition and let $\SA := \{\beta_1, \dotsc, \beta_n\}\cup \{\la_1, \dotsc, \la_s\}$ be the asserted basis of~$\SingXS$. If $\mu\in\un$ is $\setS$-singular, then its coordinates $(\varepsilon_1, \dotsc, \varepsilon_n; e_1, \dotsc, e_s)$ with respect to~$\SA$ satisfy the following conditions:
\[
(-1)^{\varepsilon_i} = \Legendre{\mu}{\gb_i}
\qquad\text{and}\qquad
e_j \equiv \ord_{\gp_j}\mu\pmod{2}
\]
for all $i\in \{1,\dotsc, n\}$ and $j\in \{1,\dotsc, s\}$
\end{prop}

\begin{proof}
Write $\mu = \beta_1^{\varepsilon_1}\dotsm \beta_n^{\varepsilon_n}\cdot \la_1^{e_1}\dotsm \la_s^{e_s}$ and fix $j\leq s$. Let $\gp_j\in \SmS$ be the place corresponding to~$\la_j$. Then $\la_j$ is the only element among $\beta_1, \dotsc, \beta_n, \la_1, \dotsc, \la_s$ that has an odd valuation at~$\gp_j$. Therefore 
\[
\ord_{\gp_j}\mu \equiv e_j\cdot \ord_{\gp_j}\la_j\equiv e_j\pmod{2}. 
\]
Take now $\mu' := \mu\cdot \la_1^{e_1}\dotsm \la_s^{e_s}$. Then $\mu'$ and $\beta_1^{\varepsilon_1}\dotsm \beta_n^{\varepsilon_n}$ are in the same square-class. This means that $\mu'$ is singular and Proposition~\ref{prop:coords} says that its coordinates must satisfy the assertion.
\end{proof}

Now its time to forge the preceding propositions into actual algorithms. We begin with the one that constructs~$\lap$.

\begin{alg}\label{alg:lambda_p}
Let $\SB = \{\beta_1, \dotsc, \beta_n\}$ be a basis of~$\SingX$ and $\GB = \{\gb_1, \dotsc, \gb_n\}$ a compatible set of places. Let $\GB' := \{\gb_{n+1}, \dotsc, \gb_m\}\subset X$, $m\geq n$ be an auxiliary set disjoint with $\GB$ and such that $\GB\cup \GB'$ generate the whole Picard group of~$X$. Further, let $\Sprime = \{\gs_1, \dotsc, \gs_s\}$  be a nonempty subset of~$X$ disjoint with~$\GB\cup \GB'$ and linearly independent in $\PPicX$. Given a place $\gp\in X$ such that $\class{\gp} + 2\PicX\in \lin_{\FF_2}\bigl\{ \class{\gs_i} + 2\PicX\st i\leq s\bigr\}$, this algorithm finds an element $\lap\in \un$ satisfying the first assertion of Proposition~\ref{prop:S-singular_basis}.
\begin{enumerate}
\item\label{st:lap:coords} Use Proposition~\ref{prop:coords} to compute the coordinates of $\gp$ and $\gs_1, \dotsc, \gs_s$ with respect to the basis~$\GB$ of $\PPicX$.
\item Using linear algebra \textup(over~$\FF_2$\textup) find $\varepsilon_1, \dotsc, \varepsilon_s\in \{0,1\}$ such that
\[
\class{\gp}\equiv \varepsilon_1\class{\gs_1} + \dotsb + \varepsilon_s\class{\gs_s}\pmod{2\PicX}.
\]
\item\label{st:lap:setT} Select the subset $\setT:= \{\gs_i\st \varepsilon_i = 1\} \subset \Sprime$ and let $t := \card{\setT}$. Denote the elements of~$\setT$ by $\gt_1, \dotsc, \gt_t$.
\item Take a lattice $\ZZ^{1+t+m}$ and a map $\psi: \ZZ^{1+t+m}\to \PicX$ given by the formula
\[
\psi\bigl( (v_0, \dotsc, v_t, v_{t+1}, \dotsc, v_{t+m})\bigr) := v_0\class{\gp} + \sum_{1\leq i\leq t} v_i\class{\gt_i} + \sum_{1\leq i\leq m} 2v_{t+i}\class{\gb_i}.
\]
\item\label{st:lap:kernel} Using linear algebra \textup(over~$\ZZ$\textup) construct a sub-lattice $V := \ker\psi$.
\item Find a vector $v = (v_0, \dotsc, v_t)\in V$ such that $v_0\odd$.
\item Let $\la$ be a generator of the Riemann--Roch space 
\[
\Ell\biggl(v_0\gp + \sum_{1\leq i\leq t}v_i\gt_i + \sum_{1\leq i\leq m}2v_{t+i}\gb_i\biggr).
\]
\item Set $\lap := \la\cdot \prod_{i\leq n} \beta_i$, where $(-1)^{e_i} = \Legendre{\la}{\gb_i}$.
\item Output $\lap$.
\end{enumerate}
\end{alg}

\begin{poc}
Assume that $v = (v_0,\dotsc, v_{t+m})$ sits in the lattice~$V$ constructed in step~\eqref{st:lap:kernel}. Then there is $\la_1\in \un$ such that
\[
\divX\la_1 = v_0\gp + \sum_{1\leq i\leq t}v_i\gt_i + \sum_{1\leq i\leq m}2v_{t+i}\gb_i.
\]
In particular $\la_1$ is $\bigl(\setT\cup \{\gp\}\bigr)$-singular, so also $\bigl(\Sprime\cup\{\gp\}\bigr)$-singular. We claim that $V$ contains an element~$v$ with an odd first coordinate. Indeed, by assumption we have
\[
\class{\gp} \equiv \sum_{i\leq s}\varepsilon_i \class{\gs_i} = \sum_{i\leq t}\class{\gt_i}\pmod{2\PicX}.
\]
Therefore, as in the proof of Proposition~\ref{prop:S-singular_basis}, there are: $\la_2\in \un$ and $\SD\in \DivX$ such that
\[
\divX\la_2 = \gp + \sum_{i\leq t}\gt_i + 2\SD.
\]
By assumption, $\GB\cup \GB'$ generate the whole Picard group. Thus the class of~$\SD$ can be written as $\class{\SD} = v_{t+1}\class{\gb_1} + \dotsb + v_{t+m}\class{\gb_m}$, for some integers $v_{t+1}, \dotsc, v_{t+m}$. This means that there is $\mu\in \un$ such that
\[
\divX\mu = -\SD + \sum_{1\leq i\leq m} v_{t+i}\gb_i.
\]
Consequently we obtain
\[
\divX(\la_2\mu^2) = (1 + 2\ord_\gp\mu)\cdot \gp + \sum_{1\leq i\leq t} (1 + 2\ord_{\gt_i}\mu)\cdot \gt_i + \sum_{1\leq i\leq m} 2v_{t+1}\gb_i
\]
and trivially $\la_2$, $\la_2\mu^2$ are in the same square class. 

Denote $\la := \la_2\mu^2$. By the preceding section, $\la$ is $\bigl(\Sprime\cup\{\gp\}\bigr)$-singular and it has an odd valuation at~$\gp$. In particular 
\[
v = (\ord_\gp\la, \ord_{\gt_1}\la, \dotsc, \ord_{\gt_t}\la, \ord_{\gb_1}\la,\dotsc, \ord_{\gb_m}\la)\in V
\]
has an odd first coordinate. This proves the claim

Now let $I\subset \{1, \dotsc, n\}$ be the set of these indices for which $\Legendre{\la}{\gb_i}= -1$. Then $\lap := \la\cdot \prod_{i\in I}\beta_i$ remains to be $\bigl(\Sprime\cup\{\gp\}\bigr)$-singular and has an odd valuation at~$\gp$, but now it is a local square at every $\gb\in \GB$. Hence it is the element we are after.
\end{poc}

\begin{rem}
The set $\GB'$ appearing in the above algorithm may be possibly empty. It is so, when $\PicZeroX$ can be decomposed into a direct sum of cyclic groups of even orders. 
\end{rem}

We are now ready to construct the group of $\setS$-singular elements. The correctness of the next algorithm follows from Proposition~\ref{prop:S-singular_basis}.

\begin{alg}\label{alg:S-singular_basis}
Given a finite \textup(possibly empty\textup) set of places $\setS\subset X$, this algorithm constructs a basis \textup(over~$\FF_2$\textup) of the group $\SingXS$ of $\setS$-singular singular elements modulo squares.
\begin{enumerate}
\item Find a basis~$\GB$ of $\PPicX$ disjoint with~$\setS$ and a set $\GB'\subset X$ disjoint with $\GB\cup \setS$ and such that $\GB\cup \GB'$ generates $\PicX$.
\item Use Algorithm~\ref{alg:singular_basis} to construct a basis~$\SB$ of~$\SingX$ compatible with~$\GB$.
\item Use Proposition~\ref{prop:coords} to compute the coordinates with respect to~$\GB$ of all $\gs\in \setS$.
\item Using these coordinates, find a maximal subset~$\Sprime$ of~$\setS$ linearly independent in $\PPicX$.
\item If $\Sprime = \setS$ \textup(in particular if $\setS$ is empty\textup) output the basis~$\SB$ and terminate.
\item Denote $V := \lin_{\FF_2}\bigl\{ \class{\gs}+2\PicX\st \gs\in \Sprime\bigr\}$ and $s := \dim V = \card{\Sprime}$.
\item For every $\gp\in \SmS$ do the following:
	\begin{enumerate}
	\item Find the coordinates $\varepsilon_1, \dotsc, \varepsilon_s\in \{0,1\}$ of~$\gp$ in~$V$ with respect to the basis~$\Sprime$.
	\item Set $\setT := \{ \gs_i\in \Sprime\st \varepsilon_i = 1\}$.
	\item Execute Algorithm~\ref{alg:lambda_p}, but skip steps (\ref{st:lap:coords}--\ref{st:lap:setT}), and construct $\lap\in \un$ that satisfies the first assertion of Proposition~\ref{prop:S-singular_basis}.
	\end{enumerate}
\item Output $\SB\cup \SL$, where $\SL := \{\lap\st \lap\in \SmS\}$.
\end{enumerate}
\end{alg}

\section{Random generation of singular elements}\label{sec:random}
Algorithms presented in the previous two sections rely on an explicit description of $\PicX$. While algorithms that compute the Picard group are known and have been implemented in existing computer algebra systems, it is also known that the whole process can take considerable amount of time. If we need just one, random singular element of~$\K$, we may do better than construct the whole group~$\SingX$.

\begin{alg}\label{alg:las_vegas_beta}
Given a global function field~$\K$, this algorithm constructs a random singular element of~$\K$.
\begin{enumerate}
\item Pick two random places $\gp, \gq$ of the same degree.
\item\label{st:k} Find $k := \min\{ j\geq 1\st \dim( j\gp - j\gq) \neq 0\}$.
\item If $k$ is even, output a generator of the Riemann-Roch space $\Ell(k\gp-k\gq)$ and terminate.
\item\label{st:constant_singular} Otherwise, let $\zeta\in \units{\Fq}$ be a non-square constant. Output at random \textup(with probability $\sfrac12$\textup) either~$\zeta$ or~$1$.
\end{enumerate}
\end{alg}

\begin{poc}
By the finiteness of $\PicZeroX$ (see \cite[Proposition~V.1.3]{Stichtenoth1993}), the class of $\SD := \gp - \gq$ is $k$-torsion for some positive integer~$k$. Therefore $k\cdot \SD$ is principal, while $i\cdot \SD$ is not for any nonzero $i < k$. Hence there is $\beta\in \un$ such that $k\cdot \SD = \divX\beta$. In addition, $\beta$ generates the Riemann--Roch space  $\Ell(k\cdot \SD)$ since this space has dimension one by \cite[Corollary~1.4.12]{Stichtenoth1993}. We claim that $\beta$ is a singular element of~$\K$. Write the divisor $\divX\beta$ of~$\beta$ in the form
\[
\divX\beta = m\gp + n\gq + \sum_{i\leq s} k_i\gr_i,
\]
where $\gr_1, \dotsc, \gr_s\in X$ are distinct from $\gp, \gq$. We have $\divX\beta \geq -\SD = k\gq - k\gp$. Therefore $m\geq -k$, $n\geq k$ and $k_i\geq 0$ for every index~$i$. In particular, $\gp$ is the only pole of~$\beta$. Furthermore, we have
\[
0 = \deg(\divX\beta) = m\cdot \deg\gp + n\cdot \deg\gq + \sum_{i\leq s} k_i\cdot \deg \gr_i.
\]
From the fact that $\gp$ and~$\gq$ have the same degree we infer that
\[
-k\cdot \deg \gp 
\geq -n\cdot \deg \gp 
= -n\cdot \deg \gq 
= m\cdot \deg\gp + \sum_{i\leq s} k_i\cdot \deg \gr_i
\geq m\cdot \deg \gp.
\]
It follows that $m$ equals $-k$. Consequently we obtain
\begin{multline*}
0
= -k\cdot \deg \gp + n\cdot \deg \gq + \sum_{i\leq s}k_i\cdot \deg \gr_i\\
= (n-k)\cdot \deg \gq + \sum_{i\leq s} k_i\cdot \deg \gr_i
\geq (n-k)\cdot \deg \gq
\geq 0.
\end{multline*}
Thus we have $n = k$ and $\divX\beta = k\gq - k\gp= -\SD$. Hence $\ord_\gr\beta \even$ for every $r\in X$, as claimed.
\end{poc}

It remains to analyze the distribution of elements of~$\SingX$ produced by the algorithm. Unfortunately the distribution depends on the structure of the Picard group of~$X$. (And in an actual implementation it is further impacted by the quality of the generator of random places of~$X$.) Thus in general the distribution does not have to be uniform, unless $\PicZeroX$ is an elementary $2$-group. 

\begin{lem}
Let $G$ be a finite abelian group of a form $G = G_1\oplus \dotsb \oplus G_n$, where every~$G_i$ is cyclic of order $\card{G_i} = 2^{k_i}\cdot t_i$ with $2\nmid t_i$. Let $g\in G$ be a random \textup(uniformly distributed\textup) element. The probability that the order of~$g$ is odd equals $2^{-(k_1 + \dotsb + k_n)}$.
\end{lem}

\begin{proof}
Write $g$ as $g = (g_1, \dotsc, g_n)$ where $g_i\in G_i$. The order of~$g$ is odd if and only if the orders of all~$g_i$ are odd. Now, $G_i$ is cyclic and $\card{G_i} = 2^{k_i}\cdot t_i$, hence $G_i$ contains precisely $t_i$ elements of odd order. Thus the probability that the order of~$g_i$ is odd equals $2^{-k_i}$ and consequently the probability that the order of~$g$ is odd equals $2^{-k_1}\dotsm 2^{-k_n}$, as claimed.
\end{proof}

The group~$\SingX$ can be expressed in a form $\{1,\zeta\}\oplus \SingX'$, where $\SingX'$ is a subgroup of~$\SingX$. Thus, probability of random picking a constant singular element equals $\sfrac{2}{\card{\SingX}}$, provided that the distribution is uniform. The next corollary gives a partial answer to the question of the distribution of elements produced by Algorithm~\ref{alg:las_vegas_beta}.

\begin{cor}\label{cor:probability}
Probability that Algorithm~\ref{alg:las_vegas_beta} outputs a constant singular element \textup(i.e. that step~\eqref{st:constant_singular} gets executed\textup) is less than or equal to $\sfrac{2}{\card{\SingX}}$.
\end{cor}

\begin{proof}
As in the proof of correctness of the algorithm, denote $\SD := \gp-\gq$, where $\gp, \gq$ are two random points of the same degree. Recall that $\PicX\cong \PicZeroX\oplus \ZZ$, where the second coordinate of the class of a divisor is just its degree. It is a known consequence of Chebotarev density theorem, that classes of places of~$X$, projected onto the first coordinate, are equidistributed in $\PicZeroX$. Hence $\class{\SD}$ may be any one of elements of $\PicZeroX$ with equal probability. The preceding lemma asserts that the probability that $\class{\SD}$ has an odd order cannot exceed $2^{-\rk\PicZeroX}$. Take a place $\go\in X$ of odd degree. By \cite[Proposition~2.3 and Lemma~2.6]{CKR2018} we have
\[
\rk\PicZeroX
= \rk\Pic(X\setminus \{\go\})
= \rk\Sing{\{\go\}} - 1 
= \rk\SingX - 1.
\]
Consequently the probability of executing step~\eqref{st:constant_singular} is bounded from above by the quantity
\[
2^{-\rk\PicZeroX} = 2^{1 - \rk\SingX} = \frac{2}{\card{\SingX}}.
\]
This proves the assertion.
\end{proof}

\bibliographystyle{plain} 
\bibliography{sms}
\end{document}